\documentclass{amsart}
\usepackage{amsmath}
\usepackage{amsfonts}
\usepackage{amssymb}
\usepackage{graphicx}
\usepackage{amsthm}
\usepackage{cite}
\usepackage{tikz}

\theoremstyle{plain}
\newtheorem{thm}{Theorem}
\newtheorem{lem}[thm]{Lemma}
\newtheorem{prop}[thm]{Proposition}

\newtheorem*{propo}{Proposition}
\newtheorem*{theom}{Theorem}

\theoremstyle{definition}
\newtheorem{defn}[thm]{Definition}

\theoremstyle{remark}

\newcommand{\inv}{^{-1}}
\newcommand{\ov}{\overline}
\newcommand{\til}{\rightarrow}

\newcommand{\sfE}{\mathsf{E}}

\begin{document}
\title{Compact metrizable structures and classification problems}
\author {Christian Rosendal and Joseph Zielinski}
\address{Department of Mathematics, Statistics, and Computer Science (M/C 249)\\
University of Illinois at Chicago\\
851 S. Morgan St.\\
Chicago, IL 60607-7045\\
USA}
\email{rosendal.math@gmail.com}
\urladdr{http://homepages.math.uic.edu/$~$rosendal}

\address{Department of Mathematical Sciences\\
Wean Hall 6113\\
Carnegie Mellon University\\
Pittsburgh, PA 15213\\
USA}
\email{zielinski.math@gmail.com}
\urladdr{http://math.cmu.edu/~zielinski/}

\begin{abstract}
We introduce and study the framework of compact metric structures and their associated notions of isomorphisms such as homeomorphic and bi-Lipschitz isomorphism. This is subsequently applied to model various classification problems in analysis such as isomorphism of $C^*$-algebras and affine homeomorphism of Choquet simplices, where among other things we provide a simple proof of the completeness of the isomorphism relation of separable, simple, nuclear $C^*$-algebras recently established by M. Sabok.
\end{abstract}

\maketitle

\section{Introduction}
The present paper deals with a general approach to determining the complexity of certain classification problems in Analyis and Topology through representing objects as compact metric or metrizable structures. The research here began with a question posed by Pierre-Emmanuel Caprace concerning the complexity of isomorphism of totally disconnected locally compact groups and culminated in the determination by the second author of the complexity of the homeomorphism relation between compact metrizable spaces \cite{Zielinski2016635}. Our goal in this paper is to present some general techniques for classifying objects via compact metric structures and along the way produce a substantially simplified proof of M. Sabok's result on the complexity of isomorphism of separable, simple, nuclear $C^*$-algebras \cite{sabok2015completeness}.

The concept of a compact metrizable or metric structure is fairly simple. Namely, given a countable relational language $\mathcal L$, a compact metrizable $\mathcal L$-structure $\mathcal M$ is a compact metrizable space $M$ equipped with closed relations $R^\mathcal M \subseteq M^n$ for every $n$-ary relation symbol $R\in \mathcal L$. The structure is metric, as opposed to metrizable, if, moreover, $M$ is given a fixed metric $d^\mathcal M$ compatible with its topology. One may now talk about isometric or bi-Lipschitz isomorphism of compact metric $\mathcal L$-structures and, similarly, homeomorphic isomorphism of compact metrizable $\mathcal L$-structures.

Many classical mathematical structures may be coded up to a relevant notion of isomorphism as compact metric or metrizable structures. For an easy example, observe that a Choquet simplex $K$ can be coded up to affine homeomorphism by the compact space $K$ along with the closed ternary relation 
$$
R=\{(x,{\frac 12}x+\frac 12y, y)\mid x,y\in K\}
$$
coding the midpoint between any two points. 

The basic fact concerning these structures is that the complexity of the relation of homeomorphic isomorphism is independent of the relational language $\mathcal L$.

\begin{propo} [See Proposition \ref{bireducible to Egrp}]
For any countable relational language $\mathcal L$,  the relation of homeomorphic isomorphism between compact metrizable $\mathcal L$-structures is Borel bireducible with the complete orbit equivalence relation $\mathsf{E}_{\mathrm{grp}} $ induced by a Polish group action.
\end{propo}

The novelty here is that the relation of homeomorphic isomorphism between compact metrizable $\mathcal L$-structures is always Borel reducible to $\mathsf{E}_{\mathrm{grp}} $. The reduction in the other direction was established in \cite{Zielinski2016635} relying, in turn, on \cite{sabok2015completeness}. Namely, in  \cite{sabok2015completeness}, Sabok shows that the complete orbit equivalence relation $\mathsf{E}_{\mathrm{grp}} $ is Borel reducible to the relation of affine homeomorphism of Choquet simplices. Subsequently, in \cite{Zielinski2016635}, a simple method is given for reducing homeomorphic isomorphism between compact metrizable $\mathcal L$-structures to homeomorphism of compact metrizable spaces, i.e., to homeomorphic isomorphism of compact metrizable structures of the empty language. The specific language $\mathcal L$ treated in in \cite{Zielinski2016635} consists of a single ternary predicate $R$ used for coding the midpoint relation in Choquet simplices, but, as is pointed out there, the method is completely general and easily applies to any other countable language. 

In Section \ref{complete}, we show how to entirely bypass the involved construction of  \cite{sabok2015completeness} to obtain a simple and direct proof of the following fact.
\begin{propo}
Let $\mathcal L$ be the language consisting of a unary predicate and two binary relation symbols. Then the complete orbit equivalence relation $\mathsf{E}_{\mathrm{grp}}$ induced by ${\rm Iso}(\mathbb U)\curvearrowright F({\rm Iso}(\mathbb U))$ is Borel reducible to homeomorphic isomorphism of compact metrizable $\mathcal L$-structures.
\end{propo}

Coupling this with \cite{Zielinski2016635} we see that the complete orbit equivalence relation is reducible to homeomorphism of compact metrizable spaces and thus isomorphism of commutative $C^*$-algebras (via the map $K\mapsto C(K)$). Moreover, the map $K\mapsto M(K)$ taking a compact metrizable space to its simplex of Borel probability measures is then a further reduction to the relation of affine homeomorphism of Choquet simplices, thereby obtaining the main result of \cite{sabok2015completeness}.

All this concerns homeomorphic isomorphism; turning instead to  isometric and bi-Lipschitz isomorphism of compact metric structures, the complexity is very different. First, a simple modification of a proof of M. Gromov \cite{gromov}, shows that, for any countable relational language $\mathcal L$, the following holds.
\begin{theom}[See Theorem \ref{isometry smooth}]
Isometric isomorphism of compact metric $\mathcal L$-structures is smooth.
\end{theom}

Also, in \cite{rosendal2005cofinal} the relation of bi-Lipschitz homeomorphism between compact metric spaces was shown to be Borel bi-reducible with the complete $K_\sigma$ equivalence relation and further improvements to this allows us to add any countable language $\mathcal L$ without altering the complexity.

\begin{theom}[See Theorem \ref{bi-Lipschitz}]
The relation of bi-Lipschitz isomorphism between compact metric $ \mathcal{L} $-structures is bi-reducible with the complete $K_{\sigma}$ equivalence relation.
\end{theom}

As an application of these ideas, we investigate different classes of Polish groups with canonical compactifications such as the locally compact and Roelcke precompact.
\begin{theom}[See Theorem \ref{main}]
The relation of topological isomorphism between locally compact and Roelcke precompact Polish groups is classifiable by compact metrizable structures and is thus Borel reducible to $\mathsf{E}_{\rm grp}$.
\end{theom}

In case the compactification turns out to be totally disconnected, i.e., a Cantor space, there is a further drop in complexity.
\begin{propo}[See Proposition \ref{non-Archimedean}]
The relation of isomorphism between locally compact or Roelcke precompact non-Archimedean Polish groups is Borel reducible to isomorphism between countable structures.
\end{propo}

Moreover, this result can be applied in a model theoretical context to obtain the following.
\begin{propo}[See Proposition \ref{bi-interpretability}]
The bi-interpretability relation for $ \omega $-categorical structures is classifiable by countable structures.
\end{propo}

The reader primarily interested in understanding the simple proof of the completeness of homeomorphism of compact metric spaces and the connection with the classification of $C^*$-algebras may restrict the attention to the short Sections \ref{prelim} and \ref{complete}. The only thing needed from Section \ref{Polish groups} to understand Section \ref{complete} is the definition of Roelcke precompactness.


\section{Preliminaries}\label{prelim}
\subsection{Compact metric spaces with closed relations}
Throughout our paper, $ \mathcal{Q} = [0,1]^{\mathbb{N}}$ denotes the Hilbert cube. Recall that, for a Polish space, $ X $, $ K(X) = \{A \subseteq X \mid A \text{ compact} \} $ is the hyperspace (or exponential space) of compact subsets of $ X $ equipped with the {\em Vietoris topology}, that is, the topology whose basic open sets have the form
$$
\{A\in K(X)\;|\; A\subseteq U\;\&\; A\cap V_1\neq \emptyset \;\&\; \ldots\;\&\; A\cap V_n\neq \emptyset\},
$$ 
where $U, V_i$ are open subsets of $X$. Alternatively, the Vietoris topology is that induced by the Hausdorff metric.

Let $ \mathcal{L} $ be a countable, relational language and, for each $R\in \mathcal L$, let $ \alpha(R) $ be a natural number called the {\em arity} of $ R $. A  {\em compact metrizable $ \mathcal{L} $-structure} is a tuple  
$$ 
\mathcal{M} =(M, (R^\mathcal M)_{R\in \mathcal L}),
$$
whose domain, $ M $, is a compact metrizable space and in which, for every symbol $ R \in \mathcal{L} $, the interpretation $ R^{\mathcal{M}} $ is a closed subset of $ M^{\alpha(R)} $.

For a fixed languange $ \mathcal{L} $  and structures $ \mathcal{M} = (M,(R^{\mathcal{M}})_{R \in \mathcal{L}}), \mathcal{N} = (N,(R^{\mathcal{N}})_{R \in \mathcal{L}}) $, we write $ \mathcal{M} \cong_{\mathcal{L}} \mathcal{N} $ when $ \mathcal{M} $ and $ \mathcal{N} $ are \emph{homeomorphically isomorphic}, i.e., when there is a homeomorphism $ f\colon M \to N $ so that for every $ R \in \mathcal{L} $, and all $ x_{1},\dots,x_{\alpha(R)} \in M $,
\[ 
R^{\mathcal{M}}(x_{1},\dots,x_{\alpha(R)}) \iff R^{\mathcal{N}}(f(x_{1}),\dots,f(x_{\alpha(R)})).  
\]
We will simply write $ \mathcal{M} \cong \mathcal{N}$ when $ \mathcal{L} $ is clear from the context.

As every compact metrizable space homeomorphically embeds into $\mathcal Q$, every compact metrizable $\mathcal L$-structure is homeomorphically isomorphic to one whose domain is a closed subset of $\mathcal Q$. We therefore let  $ \mathfrak{K}_{\mathcal{L}} $ be the set of such structures, i.e., 
 \[
\mathfrak{K}_{\mathcal{L}} = \{ (M,(R^{\mathcal{M}})_{R \in \mathcal{L}}) \in K(\mathcal{Q}) \times \prod_{R \in \mathcal{L}} K(\mathcal{Q}^{\alpha(R)}) \mid R^{\mathcal{M}} \subseteq M^{\alpha(R)}  \}.  
 \]
We observe that $\mathfrak K_\mathcal L$ is a closed subset of $K(\mathcal{Q}) \times \prod_{R \in \mathcal{L}} K(\mathcal{Q}^{\alpha(R)})$ and hence a Polish space in its own right.

We shall also consider $\cong_\mathcal L$ as a relation on $\mathfrak K_\mathcal L$ and thus make the following identification,
$$
\cong_\mathcal L=\big\{\big(\mathcal M, \mathcal N\big)\in \mathfrak K_\mathcal L\times \mathfrak K_\mathcal L\;\big|\; \mathcal M\cong \mathcal N\big\}.
$$
It is not hard to see that $\cong_\mathcal L$ is an analytic set, which will also follow from computations below.

Similarly, for a countable relational language $\mathcal L$, we  define the Polish space ${\rm Mod}(\mathcal L)$ of $\mathcal L$-structures with universe $\mathbb N$ by
$$
{\rm Mod}(\mathcal L)=\prod_{R\in \mathcal L}2^{\big(\mathbb N^{\alpha(R)}\big)}.
$$
Thus, a point $x\in {\rm Mod}(\mathcal L)$ can be identified with the structure $(\mathbb N, (x_R)_{R\in \mathcal L})$, where $(a_1, \ldots, a_{\alpha(R)})\in x_R$ if and only if $x(R)(a_1, \ldots, a_{\alpha(R)})=1$.


\subsection{Borel reducibility}
Recall that for equivalence relations $ \mathsf{E} $ and $ \mathsf{F} $ on Polish spaces $ X $ and $ Y $, respectively, $ \mathsf{E} $ is \emph{Borel reducible} to $ \mathsf{F} $ (written $ \mathsf{E} \leqslant_{B} \mathsf{F} $) when there is a Borel-measurable function $ f\colon X \to Y $ so that $ x \mathsf{E} y \iff f(x) \mathsf{F} f(y) $ for all $ x,y \in X $. When $ \mathsf{F} \leqslant_{B} \mathsf{E} \leqslant_{B} \mathsf{F}$, we write $ \mathsf{E} \sim_{B} \mathsf{F} $ (``$ \mathsf{E} $ and $ \mathsf{F} $ are \emph{Borel bi-reducible}''), and when $ \mathsf{F} \nleqslant_{B} \mathsf{E} \leqslant_{B} \mathsf{F}$ we write $ \mathsf{E} <_{B} \mathsf{F} $.

Suppose $G\curvearrowright X$ is a  continuous action of a Polish group on a Polish space $X$. This action gives rise to an orbit equivalence relation, $ \mathsf{E}_{G}^{X} $,  according to which, for any $ x,y \in X $,
\[ x \mathsf{E}_{G}^{X} y \iff \exists g \in G \; g \cdot x=y. \]

As was shown in \cite{becker1996descriptive}, for every Polish group $G$ there is a Polish group action $G\curvearrowright X$ whose orbit equivalence relation, $\mathsf{E}_G=\mathsf{E}_{G}^{X}$, is {\em complete} among all orbit equivalence relations induced by $G$, i.e., so that, for any other continuous action $ G \curvearrowright Y$ on a Polish space $Y$, one has $ \mathsf{E}_{G}^{Y} \leqslant_{B} \mathsf{E}_{G}$. Similarly, if $G$ is a universal Polish group, that is, containing every other Polish group as a closed subgroup, e.g., $G={\rm Homeo}(\mathcal Q)$ or $G={\rm Iso}(\mathbb U)$, then $\mathsf{E}_G$ is complete among all orbit equivalence relations induced by Polish groups. We let 
$\mathsf{E}_{\mathrm{grp}}$ denote some realisation of this complete orbit equivalence relation. 

In other words, if $\mathsf E_G^X$ is any orbit equivalence relation of a continuous action $G\curvearrowright X$  of a Polish group on a Polish space, then $\mathsf E_G^X\leqslant_B \mathsf E_G$, while $\mathsf E_G\leqslant_B\mathsf E_{\rm grp}$. We observe now that, up to Borel bi-reducibility, we may take any $ \cong_{\mathcal{L}} $ to be a representative for $ \mathsf{E}_{\mathrm{grp}} $.

\begin{prop} \label{bireducible to Egrp}
Let $ \mathcal{L} $ be any countable relational language. Then the isomorphism relation $ \cong_{\mathcal{L}}$ between compact metrizable $\mathcal L$-structures is Borel bireducible with $\mathsf{E}_{\mathrm{grp}} $.
\end{prop}

\begin{proof}
First observe that the relation of homeomorphic isomorphism in the empty language, $ \cong_{\emptyset} $, is simply that of homeomorphism between compact metrizable spaces, and so $ \cong_{\emptyset} \; \sim_{B} \mathsf{E}_{\mathrm{grp}} $ by the main result of \cite{Zielinski2016635}. Furthermore, for any $ \mathcal{L} $, the map $ K(\mathcal{Q}) \to \mathfrak{K}_{\mathcal{L}} $ taking $ M \mapsto (M,(M^{\alpha(R)})_{R \in \mathcal{L}}) $ (the structure, with domain $ M $, for which $ R $ holds on every $ \alpha(R) $-tuple) is clearly a reduction, showing that $ \cong_{\emptyset} \; \leqslant_{B} \; \cong_{\mathcal{L}} $.

It then remains to show that $ \cong_{\mathcal{L}} $ is below a group action. Let $ \iota : \mathcal{Q} \to \mathcal{Q} $ be $ \iota((x_{n})_{n \in \mathbb{N}}) = (\frac{1}{2}x_{n} + \frac{1}{4})_{n \in \mathbb{N}}$. The particulars of this map are unimportant, but rather we note the following two features: First, $ \iota $ is a homeomorphic embedding, and second, its image satisfies $ \iota[\mathcal{Q}] \subseteq (0,1)^{\mathbb{N}} \subseteq \mathcal{Q} $. Thus, by work of R.D. Anderson (see \cite{van2001infinite} Chapter 5), for all compact $ A, B \subseteq \iota[\mathcal{Q}] $, every homeomorphism between $ A $ and $ B $ extends to an element of the homeomorphism group $ \operatorname{Homeo}(\mathcal{Q}) $.

Now, for each $ n \in \mathbb{N} $, $ \iota $ induces a map $ \iota^{n}\colon  \mathcal{Q}^{n} \to \mathcal{Q}^{n} $ by $ \iota^{n}(x_{1},\dots,x_{n}) = (\iota(x_{1}),\dots,\iota(x_{n})) $. Each $ \iota^{n} $, in turn, induces a map $ \iota^{n}_{*}\colon  K(\mathcal{Q}^{n}) \to K(\mathcal{Q}^{n}) $ by $ A \mapsto \iota^{n}[A] $. Finally, let
\[
  \iota_{\mathcal{L}} = \iota^1_{*} \times \prod_{R \in \mathcal{L}} \iota^{\alpha(R)}_{*} \colon \mathfrak{K}_{\mathcal{L}} \to \mathfrak{K}_{\mathcal{L}}. 
\]

Similarly, every $ g \in \operatorname{Homeo}(\mathcal{Q}) $ induces maps $ g^{n}$, $ g^{n}_{*} $, and $ g_{\mathcal{L}} $, the latter of which determines a natural action of $ \operatorname{Homeo}(\mathcal{Q}) $ on $ K(\mathcal{Q}) \times \prod_{R \in \mathcal{L}} K(\mathcal{Q}^{\alpha(R)}) $, for which $ \mathfrak{K}_{\mathcal{L}} $ is an invariant subspace. We claim that $ \iota_{\mathcal{L}} $ is a reduction from $ \cong_{\mathcal{L}} $ to the orbit equivalence relation of this action $ \operatorname{Homeo}(\mathcal{Q}) \curvearrowright \mathfrak{K}_{\mathcal{L}} $. 

Indeed, if $ g \in \operatorname{Homeo}(\mathcal{Q}) $ is such that $ g_{\mathcal{L}}(\iota_{\mathcal{L}}(\mathcal{M})) = \iota_{\mathcal{L}}(\mathcal{N}) $, then $ \iota^{-1} \circ g \circ \iota $ is a homeomorphic isomorphism between $ \mathcal{M} $ and $ \mathcal{N} $. Conversely, if $ h\colon M \to N $ determines a homeomorphic isomorphism $ \mathcal{M} \to \mathcal{N} $, then $ (\iota \circ h \circ \iota^{-1})_{\mathcal{L}} $ is a homeomorphic isomorphism between $ \iota_{\mathcal{L}}(\mathcal{M}) $ and $ \iota_{\mathcal{L}}(\mathcal{N}) $. But as the domain of $ \iota_{\mathcal{L}}(\mathcal{M}) $ and $ \iota_{\mathcal{L}}(\mathcal{M}) $ are $ \iota[M] $ and $ \iota[N] $, and both are subsets of $ \iota[\mathcal{Q}] $, $ \iota \circ h \circ \iota^{-1} $ extends to a $ g \in \operatorname{Homeo}(\mathcal{Q}) $. But $ g_{\mathcal{L}}(\iota_{\mathcal{L}}(\mathcal{M})) = (\iota \circ h \circ \iota^{-1})_{\mathcal{L}}(\iota_{\mathcal{L}}(\mathcal{M})) = \iota_{\mathcal{L}}(\mathcal{N}) $. So $ \iota_{\mathcal{L}}(\mathcal{M}) $ and $ \iota_{\mathcal{L}}(\mathcal{N}) $ are orbit equivalent.
\end{proof}

On other other hand,  the isomorhism relation between countable $\mathcal L$-structures, i.e., between points in ${\rm Mod}(\mathcal L)$, is induced  by the canonical action of $S_\infty$ on ${\rm Mod}(\mathcal L)$ and thus is Borel reducible to $\mathsf E_{S_\infty}$.


\section{Polish groups with canonical compactifications} \label{Polish groups}

In this section, we begin with a family of Polish groups, namely, locally compact or Roelcke precompact groups,  and show how, for an appropriate $ \mathcal{L} $, one may---in a Borel manner---assign compact metrizable structures as complete invariants to these groups.

\begin{thm}\label{main}
The relation of topological isomorphism between locally compact and Roelcke precompact Polish groups is classifiable by compact metrizable structures and is thus Borel reducible to $\mathsf{E}_{\rm grp}$.
\end{thm}

By contract, the complexity of isomorphism between all Polish groups is maximum among all analytic equivalence relations. Indeed, in \cite{ferenczi2009complexity}, it is shown that the complexity of isomorphism between separable Banach spaces is complete analytic. And as two Banach spaces are linearly isomorphic if and only if the underlying additive topological groups are isomorphic, the same holds for isomorphism of Polish groups.

\subsection{The reduction}
We begin by first showing how to assign compact metrizable structures as complete invariants and will subsequently verify that this assignment may be done in a sufficiently constructive manner in order to be Borel.

\subsubsection{Locally compact groups}
The case of locally compact  groups is essentially trivial. First of all, every compact Polish group $G$ is canonically a compact metrizable structure, by viewing $G$ as the tuple $(G, {\rm Mult_G})$, where 
$$
{\rm Mult}_G=\{(g,f,h)\in G^3\mid gf=h\}
$$
is the graph of multiplication in $G$. In this way, isomorphisms of compact groups $G$ and $H$ are simply isomorphisms of the corresponding  compact metrizable structures $(G,{\rm Mult}_G)$ and $(H, {\rm Mult}_H)$.

Now, suppose instead that $G$ is a locally compact non-compact Polish group and let $ M_{G} $ denote the Alexandrov one-point compactification $ G^{*} = G\cup \{*^{\mathcal M_G}\}$. In this case, ${\rm Mult}_G$ is no longer  compact in $M_G^3$, so instead we consider its closure
$$ 
R^{\mathcal{M}_{G}} =\overline{{\rm Mult}_G}^{M_G^{3}}.
$$
Observe that, since ${\rm Mult}_G$ is closed in $G^3$, we  have that ${{\rm Mult}_G}=R^{\mathcal{M}_{G}}\cap G^3$.
It is now easily seen that the structure $\mathcal M_G=(M_G, R^{\mathcal{M}_{G}}, *^{\mathcal{M}_{G}})$ is a complete invariant for $G$, i.e., that two locally compact non-compact Polish groups $G$ and $H$ are isomorphic if and only if their associated structures are homeomorphically isomorphic.

Indeed, suppose $ \varphi: G \to H $ is a topological group isomorphism. Then $ \varphi $ extends to a homeomorphism $ G^{*} \to H^{*} $ mapping $*^{\mathcal M_G}$ to $*^{\mathcal M_H}$ and $R^{\mathcal{M}_{G}}$ to $R^{\mathcal{M}_{H}}$. So also $\mathcal M_G$ and $\mathcal M_H$ are isomorphic. 

Conversely, if $\mathcal M_G$ and $\mathcal M_H$ are isomorphic via some homeomorphism $\varphi\colon M_G\to M_H$, then $\phi$ maps $*^{\mathcal M_G}$ to $*^{\mathcal M_H}$ and thus maps $G=M_G\setminus \{*^{\mathcal M_G}\}$ to $H=M_H\setminus \{*^{\mathcal M_H}\}$. Since also $\varphi$ maps $R^{\mathcal{M}_{G}}$ to $R^{\mathcal{M}_H}$ and ${{\rm Mult}_G}=R^{\mathcal{M}_{G}}\cap G^3$, ${{\rm Mult}_H}=R^{\mathcal{M}_{H}}\cap H^3$, we find that $\varphi$ maps ${\rm Mult}_G$ to ${\rm Mult}_H$ and thus is an isomorphism of topological groups.

As we have restricted ourselves to relational languages, $\mathcal M_G=(M_G, R^{\mathcal{M}_{G}}, *^{\mathcal{M}_{G}})$ is strictly speaking not a compact metrizable structure. But this can easily be repaired simply by replacing the constant $*^{\mathcal{M}_{G}}$ be a unary predicate holding exactly at this single pont.

\subsubsection{Roelcke precompact groups}
Every topological group comes equipped with several natural uniform structures \cite{roelcke1981uniform}, including the left, right, two-sided (or upper, or bilateral), and the Roelcke (or lower) group uniformities. The {\em Roelcke uniformity} is that generated by entourages of the form 
$$ 
\{(g,h) \in G^{2} \mid h \in VgV  \} 
$$ 
as $ V $ varies over symmetric identity neighborhoods. Alternatively, if $d$ is any compatible left-invariant metric on $G$, then 
\[
d_{\wedge}(g,h) = \inf_{f \in G} \max \{d(f,g),d(f^{-1},h^{-1})\} 
\]
is a compatible metric for the Roelcke uniformity. The \emph{Roelcke precompact} groups are precisely those groups whose completions with respect to this latter uniformity are compact. Equivalently, these are the groups, $ G $, for which, given any open identity neighborhood $ V $, there is a finite $ F \subseteq G $ so that $ G = VFV $. Similarly, if $G$ is metrizable, it is Roelcke precompact if and only if $d_\wedge$ is totally bounded on $G$.

Now, for a Roelcke precompact Polish group $G$, let $ M_{G} $ be the completion of $ G $ with respect to the Roelcke uniformity and let, as before, $R^{\mathcal{M}_{G}} =\overline{{\rm Mult}_G}^{M_G^{3}}$ denote the closure in $M_G$ of the graph of the group multiplication. So $M_G$ is a compact metrizable space and $R^{\mathcal M_G}\subseteq M_G^3$ a closed subset. Set $\mathcal M_G=(M_G, R^{\mathcal M_G})$.

Now, if $ \varphi\colon  G \to H $ is a topological group isomorphism, it is then by definition a homeomorphism and moreover induces a bijection between symmetric, open, identity neighborhoods. By the description of the Roelcke uniformity above, it is therefore a bijection between entourages of the Roelcke uniformity, i.e., a uniform homeomorphism.

As such,  if $G$ and $H$ are Roelcke precompact, $\varphi$  extends uniquely  to a homeomorphism $ \varphi\colon  M_{G} \to M_{H} $, and as $ \varphi^{3}\colon M_{G}^{3} \to M_{H}^{3} $ maps the graph of multiplication in $ G $ surjectively to the graph of multiplication in $ H $, it extends to a map between their respective closures, namely it maps $ R^{\mathcal{M}_{G}} $ onto $ R^{\mathcal{M}_{H}} $, and so $ \mathcal{M}_{G} \cong \mathcal{M}_{H} $.

Conversely, suppose $ \psi\colon  \mathcal{M}_{G} \to \mathcal{M}_{H} $ is a homeomorphic isomorphism.  Letting $G' = \psi[G] \subseteq M_{H}$, we see that the set ${\rm Mult}_{G'}=\psi^3[{\rm Mult}_G]$ defines the graph of a group multiplication on $G'$. Also, as  ${\rm Mult}_G=R^{\mathcal{M}_{G}}\cap G^3$, we have
$$
{\rm Mult}_{G'}=\psi^3[{\rm Mult}_G]=R^{\mathcal{M}_{H}} \cap (G')^{3}.
$$
Moreover, as $ \psi $ is a homeomorphism, the topology $ G' $ inherits from $ M_{H} $ is the same as the one induced by $ G $ via $ \psi $, and so makes $ G' $ into a Polish topological group, and isomorphic to $ G $ as such.

As $G$ is a Polish group, it must be $G_\delta$ and dense in $M_{G}$, whence also $G'$ is dense $G_\delta$ in $M_H$. Similarly,  $H$ dense $G_\delta$ in $M_H$. Therefore $H\cap G'$ is comeagre in $M_{H}$, and in both $ G' $ and $ H $ as well.

We claim that, for elements $f,h\in H\cap G'$, the product $hf^{-1}$ is independent of whether it is calculated in $H$ or in $G'$. To see this, let $x\in H$  and $y\in G'$ be the two results of these calculations performed respectively in $H$ and in $G'$, i.e., so that $(x, f,h)\in {\rm Mult}_{H}$, while $(y, f,h)\in {\rm Mult}_{G'}$. It follows that both $(x,f,h)\in {\rm Mult}_H\subseteq R^{\mathcal M_H}$ and $(y, f,h)\in {\rm Mult}_{G'}\subseteq R^{\mathcal M_H}$. 
As $R^{\mathcal M_H}$ is the closure of ${\rm Mult}_H$, there are thus tuples $(k_n, f_n, h_n)\in {\rm Mult}_H$ converging to $(y,f,h)$, i.e., $k_n\to y$, $f_n\to f$ and $h_n\to h$. However, from $(k_n, f_n, h_n)\in {\rm Mult}_H$, we see that $k_n=h_nf_n^{-1}$ as calculated in $H$. So, by the continuity of the group operations in $H$, we have that $k_n=h_nf_n^{-1}\to x$, i.e., $x=y$.

It follows that the intersection $H\cap G'$ is closed under the operation $(f,h)\mapsto hf^{-1}$ and so, being non-empty, is  both a subgroup of $H$ and a subgroup of $G'$. But, since it is a comeagre subgroup, all of its cosets in $H$ and in $G'$ must be comeagre too, which means that there can only be one coset, i.e., $H=H\cap G'=G'$. Therefore, $ \psi[G] = G' = H $, and so $ \psi \upharpoonright G $ is a topological group isomorphism between $ G $ and $ H $.

Summing up, we arrive at the following.
\begin{prop}
For a Roelcke precompact Polish group $G$,  let $M_G$ be the completion of $G$ in the Roelcke uniformity and $R^{\mathcal M_G}= \overline{{\rm Mult}_G}^{M_G^{3}}$ denote the closure in $M_G$ of the graph of the group multiplication on $G$.

Then two Roelcke precompact Polish groups $G$ and $H$ are isomorphic as topolo\-gical groups if and only if the compact metrizable structures $\mathcal M_G=(M_G, R^{\mathcal M_G})$ and $\mathcal M_H=(M_H, R^{\mathcal M_H})$ are homeomorphically isomorphic.
\end{prop}


\subsection{Definability of the reduction}
Our goal in this section is now to show that the reductions described above may be made in a Borel manner, i.e., that they correspond to Borel reductions from appropriate standard Borel spaces of Polish groups to the space $\mathfrak K_{\mathcal L}$ of compact metrizable structures over a countable language $\mathcal L$.

\subsubsection{Parametrisations of Polish groups}
Choose a universal Polish group $\mathbb G$ such as $ \operatorname{Iso}(\mathbb{U}) $ or $ \operatorname{Homeo}(\mathcal{Q}) $, i.e., containing every Polish group, up to isomorphism, as a closed subgroup. Fix also a compatible, left-invariant metric $ d $ on $ \mathbb{G} $ of diameter $ 1 $. As is well-known, the family ${\bf GRP}$ of closed subgroups of $\mathbb G$ is a Borel subset of the Effros--Borel space $F(\mathbb G)$ of closed subsets of $\mathbb G$, that is, equipped with the standard Borel structure generated by the sets $\{F\in F(\mathbb G)\mid F\cap U\neq \emptyset\}$, where $U$ varies over open subsets of $\mathbb G$. So ${\bf GRP}$ is a standard Borel space parametrizing  the family of all Polish groups. 

 We note that as the elements of $ \mathbf{GRP} $ inherit the group operation and metric from $ \mathbb{G} $, both metric properties and group-theoretic properties of elements of $ \mathbf{GRP} $ may be described in a uniform, Borel manner. For example, we may show that the classes of compact $\bf {CG}$, locally compact $\bf {LCG}$ and Roelcke precompact $\bf {RPC}$ are Borel subsets of $\bf{GRP}$.
 
For this, we choose a sequence $ (s_{n})_{n \geq 0} $ of  Kuratowski--Ryll-Nardzewski selectors, that is, Borel functions $ s_{n}\colon  F(\mathbb{G}) \to \mathbb{G} $ so that, for non-empty $F\in F(\mathbb G)$, $\{s_n(F)\}$ is a dense subset of $F$  as in \cite{kechris1995classical} Theorem 12.3.  

Then $G\in \bf{GRP}$ is compact if and only if the metric space $(G,d)$ is totally bounded, i.e., exactly when
$$
\forall \epsilon\in \mathbb Q_+\; \exists n\; \forall k\; \exists i\leqslant n \; d(s_k(G), s_i(G))<\epsilon.
$$
Similarly, $G$ is locally compact if and only if 
$$
\exists \delta\in \mathbb Q_+\;\forall \epsilon\in \mathbb Q_+\; \exists n\; \forall k\; \big(d(s_k(G),1)<\delta\Rightarrow \exists i\leqslant n \; d(s_k(G), s_i(G))<\epsilon\big).
$$
Finally, $G$ is Roelcke precompact if and only if the corresponding Roelcke metric 
$$
d_\wedge^G(f,g)=\inf_{h\in G}\max\{d(f,h),d(h^{-1},g^{-1})\}
$$
is totally bounded on $G$, i.e., if 
$$
\forall \epsilon\in \mathbb Q_+\; \exists n\; \forall k\; \exists i\leqslant n \; \exists j\; \big(d(s_k(G), s_j(G))<\epsilon \;\&\; d(s_j(G)^{-1}, s_i(G)^{-1})<\epsilon\big).
$$
This shows that $\bf {CG}$, $\bf {LCG}$ and $\bf{RPC}$ are all three Borel in $\bf {GRP}$.

\subsubsection{Definability  of compatible metrics}
Observe that, both for locally compact non-compact and Roelcke precompact Polish groups, the universe $M_G$ of the associated compact metrizable structure is the completion of $G$ with respect to a compatible metric for the 
Alexandrov, respectively, Roelcke compactification. We now verify that such metrics can be computed in a Borel manner from $G$.

\begin{lem}\label{metric}
For $G\in \bf {LCG}\setminus \bf {CG}$, respectively $G\in \bf{RPC}$, we may choose a compatible metric $d^G_*$ on the Alexandrov compactification, respectively a compatible metric $d^G_\wedge$ on the Roelcke completion, in such a way that, for $n,m,k\in \mathbb N$, the maps
$$
d^G_*(s_n(G), s_k(G)), \quad d^G_*(s_n(G)s_m(G), s_k(G)), \quad d^G_*(*, s_k(G)),
$$
and 
$$
d^G_\wedge(s_n(G), s_k(G)), \quad d^G_\wedge(s_n(G)s_m(G), s_k(G))
$$
are all Borel in the variable $G$.
\end{lem}

\begin{proof}
Consider first $G\in \bf {LCG}\setminus \bf {CG}$. Let $k_G\geqslant 1$ be the minimal integer so that $V_G=\{g\in G\mid d(g,1)< \frac 1{k_G}\}$ is relatively compact in $G$. We define a compatible left-invariant proper metric $\partial^G$ on $G$ by  setting 
$$
\partial^G(f,g)=\inf\big( \sum_{i=1}^n w(x_i)\mid g=fx_1x_2\cdots x_n\;\&\; x_i\in V_G \cup\{s_m(G)\}_m\big),
$$
where $w(x)= d(x,1)$ for $x\in V_G$ and $w(x)=m$, where $m$ is minimal such that $x=s_m(G)$ otherwise.

We may then apply the construction in \cite{mandelkern1989metrization} to obtain a compatible metric $d^G_*$ on the Alexandrov compactification $G\cup \{*\}$. Concretely, let $\ell^G(g)=\frac 1{1+\partial^G(g,1)}$ and set
$$
d^G_*(f,g)=\min\{\partial^G(f,g), \ell^G(f)+\ell^G(g)\}
$$
and $d^G_*(*,g)=\ell^G(g)$. It is now straightforward to check that the appropriate functions have analytic graphs and hence are all Borel measurable.

For $G\in \bf {RPC}$, we immediately see that, e.g., 
$$
d^G_\wedge(s_n(G), s_k(G))=\inf_{m\in \mathbb N}\max \big\{ d(s_n(G), s_m(G)), d(s_m(G)^{-1}, s_k(G)^{-1})\big\},
$$ 
is Borel in the variable $G$.
\end{proof}

\subsubsection{Definability of the mapping}
Recall that, if $(X,d)$ is a metric space of dia\-meter $\leqslant 1$ with a dense sequence $(x_n)_n$, then the map 
$\phi\colon x\mapsto (d(x,x_n))_{n=1}^\infty$ defines a homeomorphic embedding of $X$ into the Hilbert cube $[0,1]^\mathbb N$. Moreover, $\phi$ is uniformly continuous with respect to the unique compatible uniformity on the compact space $[0,1]^\mathbb N$, e.g., induced by the metric $\partial(a,b)=\sum_n\frac {|a_n-b_n|}{2^n}$. It therefore follows that, if $(X,d)$ is totally bounded, then $\phi$ extends continuously to the completion $\overline{(X,d)}$ and thus maps $\overline{(X,d)}$ homeomorphically onto the closure $\overline {\phi[X]}$ inside $[0,1]^\mathbb N$.

Let the language $\mathcal L$ consist of a ternary relation symbol $R$. We must now show how to construct a Borel map $\mathbf {RPC}\to \mathfrak K_{\mathcal L}$ that to each Roelcke precompact closed subgroup of $\mathbb G$ computes a  homeomorphically  isomorphic copy of the  invariant $\mathcal M_G=(M_G, R^{\mathcal M_G})$ defined in the preceeding section. For this, we use the metric $d_\wedge^G$ given by Lemma \ref{metric}. Indeed, for all $n,m$, define Borel functions $\phi_n, \phi_{n,m}\colon {\bf RPC}\to [0,1]^\mathbb N$ by
$$
\phi_n(G)=\big(d^G_\wedge\big(s_n(G), s_k(G)\big)\big)_{k=1}^\infty, \quad 
\phi_{n,m}(G)=\big(d^G_\wedge\big(s_n(G)s_m(G), s_k(G)\big)\big)_{k=1}^\infty.
$$
Then $\phi_n(G)$ is simply the image in $[0,1]^\mathbb N$ of the point $s_n(G)$ under the embedding $\phi$ defined by the dense sequence $(s_k(G))_k$ in $G$. Similarly, $\phi_{n,m}(G)$ is the image of $s_n(G)s_m(G)$.
We set
$$
M_G=\overline{\big\{\phi_n(G)\in [0,1]^{\mathbb N}\;\big|\; n\in \mathbb N\big\}}
$$
and  observe that, since the conditions $M_G\cap U\neq \emptyset$ are Borel in $G$ for all open $U\subseteq [0,1]^\mathbb N$, the map $G\mapsto M_G$ is Borel.  Similarly, setting 
$$
R^{\mathcal M_G}
=\overline{
\big\{\big( \phi_n(G), \phi_m(G), \phi_{n,m}(G)\big)
\in \big( [0,1]^{\mathbb N}\big)^3
\;\big|\; 
n,m\in \mathbb N
\big\}
},
$$
$G\mapsto R^{\mathcal M_G}$ is Borel.

Now, $M_G$ is clearly homeomorphic to the Roelcke completion of $G$, while, as the set of $(s_n(G), s_m(G), s_n(G)s_m(G))$ is dense in ${\rm Mult}_G$, the relation $R^{\mathcal M_G}$ is similarly the closure in $M_G$  of the image of the graph of the multiplication. It follows that $\mathcal M_G=(M_G, R^{\mathcal M_G})\in \mathfrak K_\mathcal L$ is homeomorphically isomorphic to the invariant defined in the preceeding section.

For the case of locally compact non-compact $G$, our language $\mathcal L$ now has an additional symbol $*$ for the point at infinity in the Alexandrov compactification, which we may take to be a unary predicate $P$ holding exactly at this point. The set $M_G$ and the relation $R^{\mathcal M_G}$ are defined as in the Roelcke precompact case, except that we use $d^G_*$ in place of the metric $d^G_\wedge$. Now to see that $G\mapsto P^{\mathcal M_G}$ is Borel, we simply note that
$$
P^{\mathcal M_G}\;\cap\;\;\prod_k\;]\alpha_k,\beta_k[\;\;\neq \emptyset \;\;\;\Leftrightarrow\;\;\; \forall k\;\; \alpha_k<d_*^G(*, s_k(G))<\beta_k
$$
for any $\alpha_k, \beta_k\in \mathbb R$.

These computations conclude the proof of Theorem \ref{main}.

\section{Reducing a complete Polish group action}\label{complete}

In the proof of Proposition \ref{bireducible to Egrp}, we have seen that for every countable language $ \mathcal{L} $, the relation of homeomorphic isomorphism of compact $ \mathcal{L} $-structures is reducible to the orbit equivalence relation induced by a Polish group action. Conversely, by the main result of \cite{Zielinski2016635}, even for $\mathcal L=\emptyset$, the relation of homeomorphic isomorphism is complete for the class of all orbit equivalence relations. The argument there proceeds as follows: first it is established that for a certain $ \mathcal{L} $, homeomorphic isomorphism of compact metrizable $ \mathcal{L} $-structures is a complete orbit equivalence relation. Then it is seen that this relation reduces to that of homeomorphism between compact metrizable spaces without any additional structure.

There, this first step is achieved by representing the affine structure of a compact, convex subset of $ \mathcal{Q} $ as a compact metrizable structure. Thus, the completeness of the homeomorphic isomorphism relation might appear to depend on the completeness of the affine homeomorphism relation of these convex sets, established in \cite{sabok2015completeness}, or on the corresponding result about the isometry of separable complete metric spaces \cite{gao2003classification} from which it follows, in turn.

In this section, we adapt the arguments of Section \ref{Polish groups} to produce a reduction from a complete orbit equivalence relation to homeomorphic isomorphism between compact metrizable structures that obviates the reliance on the completeness of these other relations. And in fact, the completeness of both these relations, as well as others like the isomorphism of separable $C^{*} $-algebras \cite{sabok2015completeness} can now be seen as consequences of this fact for the homeomorphisms of compacta.

For this, recall that the prototypical example of a complete orbit equivalence relation was given in \cite{becker1996descriptive}: it is the orbit equivalence relation of a universal Polish group, $ \mathbb{G} $, acting by left translation on the space of its closed subsets, $ F(\mathbb{G}) $. Let us take $ \mathbb{G} $ to be $ \operatorname{Iso}(\mathbb{U}) $---this choice is not necessary for such a reduction, but it does simplify the language $ \mathcal{L} $ of the resulting structures. Specifically, let $ \mathcal{L} $ be a language with two binary relations and one unary predicate.

\begin{prop} \label{group action to compact structures}
	The orbit equivalence relation, $ \sfE_{\mathrm{grp}} $, of $ \operatorname{Iso}(\mathbb{U}) \curvearrowright F(\operatorname{Iso}(\mathbb{U})) $ is reducible to homeomorphic isomorphism of compact metrizable structures.
\end{prop}

\begin{proof}
	In \cite{solecki2005extending}, it is shown that $ \mathbb{G} = \operatorname{Iso}(\mathbb{U}) $ is \emph{topologically 2-generated}, i.e., there are $ a,b \in \mathbb{G} $ so that $ a $ and $ b $ generate a dense subgroup, $ \langle a,b \rangle $. Let $ R_{a} = \{ (g,ga) \in \mathbb{G}^{2} \mid g \in \mathbb{G} \} $ be the graph of right-multiplication by $ a $, and likewise let $ R_{b} $ be the graph of right multiplication by $ b $. Recall that the isometry group of the Urysohn \emph{sphere}, $ \operatorname{Iso}(\mathbb{U}_{1}) $, is also a universal Polish group, and moreover is Roelcke precompact\cite{uspenskij2008subgroups,rosendal2009topological}. Fix a topological group embedding $ \mathbb{G} \hookrightarrow \operatorname{Iso}(\mathbb{U}_{1}) $, and let $ X = \overline{\operatorname{Iso}(\mathbb{U}_{1})}^{\wedge} $ be the Roelcke compactification. Identify $ \mathbb{G} $ with its embedded copy in $ X $. Then the map $ F(\mathbb{G}) \ni A \mapsto (\overline{\mathbb{G}},\overline{R_{a}}, \overline{R_{b}}, \overline{A}) $ is a reduction, where the closures are taken in $ X $ and $ X^{2} $.
	
	For suppose $ A, B \in F(\mathbb{G}) $ with $ A \: \sfE_{\mathrm{grp}} \: B $. Then there is an $ f \in \mathbb{G} $ so that $ fA = B $. The left multiplication action of $ \operatorname{Iso}(\mathbb{U}_{1}) $ on itself extends to an action $ \operatorname{Iso}(\mathbb{U}_{1}) \curvearrowright X $, and so viewing $ f $ as an element of $ \operatorname{Iso}(\mathbb{U}_{1}) $, $ f $ extends to a homeomorphism $ \lambda_{f}: X \to X $. As $ \lambda_{f}[\mathbb{G}] = \mathbb{G} $, its closure is also fixed set-wise by $ \lambda_{f} $. Likewise, as $ fA = B $, $ \lambda_{f}[\overline{A}] = \overline{B} $. Moreover,
	\[  (\lambda_{f} \times \lambda_{f})[R_{a}] = \{(fg,fga) \in \mathbb{G}^{2} \mid g \in \mathbb{G} \} = \{(g,ga) \in \mathbb{G}^{2} \mid f^{-1} g \in \mathbb{G} \} = R_{a}  \]
	Therefore also $ (\lambda_{f} \times \lambda_{f})[\overline{R_{a}}] = \overline{R_{a}} $ and the same holds, mutatis mutandis, for $ \overline{R_{b}} $. Therefore, $ \lambda_{f} \upharpoonright \overline{\mathbb{G}} $ is a homeomorphic isomorphism $ (\overline{\mathbb{G}},\overline{R_{a}}, \overline{R_{b}}, \overline{A}) \to (\overline{\mathbb{G}},\overline{R_{a}}, \overline{R_{b}}, \overline{B}) $.
	
	On the other hand, suppose $ \rho $ is a homeomorphic isomorphism $ (\overline{\mathbb{G}},\overline{R_{a}}, \overline{R_{b}}, \overline{A}) \to (\overline{\mathbb{G}},\overline{R_{a}}, \overline{R_{b}}, \overline{B}) $. For a binary relation, $ R $, let $ R^{-1} = \{(x,y) \mid (y,x) \in R \} $. Then
	\[ R_{a^{-1}} = \{(g,ga^{-1}) \in \mathbb{G}^{2} \mid g \in \mathbb{G} \} = \{(ga,g) \in \mathbb{G}^{2} \mid g \in \mathbb{G} \} = R_{a}^{-1}, \]
	and so $ \overline{R_{a^{-1}}} = \overline{R_{a}^{-1}} = \big(\overline{R_{a}}\big)^{-1} $. Thus $ (\rho \times \rho)[\overline{R_{a^{-1}}}] = \overline{R_{a^{-1}}} $, and likewise for $ b^{-1} $. Also, note that for any $ g \in \mathbb{G} $, $ \mathbb{G}^{2} \cap \overline{R_{g}} = R_{g} $, as $ R_{g} $ is closed in $ \mathbb{G}^{2} $. Let $ C' = \rho^{-1}[\mathbb{G}] \cap \mathbb{G} $. As $ \mathbb{G} $ is comeagre in its closure, $ C' $ is comeagre in $ \mathbb{G} $, and therefore so is $ C = \bigcap_{w \in \langle a,b \rangle} C'w $. Observe that $ Cw = C $ for all $ w \in \langle a,b \rangle $.
	
	Suppose $ g \in C $ and $ s \in \{a,b,a^{-1},b^{-1}\} $. Then $ (\rho(g),\rho(gs)) = (\rho \times \rho)(g,gs) \in (\rho \times \rho)[\overline{R_{s}}] = \overline{R_{s}} $. Moreover,
	\[ \rho(g) \in \rho[C] \subseteq  \mathbb{G} \]
	and
	\[ \rho(gs) \in \rho[Cs] = \rho[C] \subseteq  \mathbb{G}. \]
	So for all $ g \in C $ and $ s \in \{a,b,a^{-1},b^{-1}\} $, $ (\rho(g),\rho(gs)) \in \overline{R_{s}} \cap \mathbb{G}^{2} = R_{s} $, and so $ \rho(gs) = \rho(g)s $. As $ Cw = C $ for all $ w \in \langle a,b \rangle $, by induction on length, $ \rho(gw) = \rho(g)w $ for all such $ w $.
	
	Now fix $ g \in C $, and let $ h \in \mathbb{G} $ be arbitrary. Take $ (w_{n})_{n \in \mathbb{N}} \subseteq \langle a,b \rangle $ with $ w_{n} \to h $. Then
	\[ \rho(gh) = \lim_{n} \rho(gw_{n}) = \lim_{n} \rho(g)w_{n} = \rho(g)h \]
	where the first and third equalities are consequences of the continuity of $ \rho $ and of multiplication in $ \mathbb{G} $.
	
	Pick any $ g \in C $. Then as $ \rho(g) \in \mathbb{G} $, and by the above, $  \rho(1_{\mathbb{G}}) = \rho(gg^{-1}) = \rho(g)g^{-1} $. So $ \rho(1_{\mathbb{G}}) $ is a product of two elements of $ \mathbb{G} $, and so in $ \mathbb{G} $ itself. Next fix $ h \in \mathbb{G} $, and let $ (g_{n})_{n \in \mathbb{N}} $ be a sequence from $ C $ converging to $ 1_{\mathbb{G}} $. Then as before,
	\[ \rho(h) = \rho(1_{\mathbb{G}}h) = \lim_{n} \rho(g_{n}h) = \lim_{n} \rho(g_{n})h = \rho(1_{\mathbb{G}})h. \]
	
	Therefore, on $ \mathbb{G} $, $ \rho $ agrees with the left translation by $ \rho(1) $, and so it must be the unique extension of this map. In particular, $ \rho[\mathbb{G}] = \mathbb{G} $, and so $ \rho(1)A = \rho[A] = \rho[\overline{A} \cap \mathbb{G}] = \overline{B} \cap \mathbb{G} = B $. So $ A \: \mathsf{E}_{\mathrm{grp}} \: B $.
\end{proof}
	
\section{Polish heaps}
As we are interested in translating subsets of Polish groups, it is useful to investigate the concept of abstract cosets (see \cite{baer1929einfuhrung,certaine1943ternary} for some early papers on the subject), that is, structures in which left and right-translations become actual automorphisms.
\begin{defn}
	A {\em heap} (a.k.a. {\em groud} or {\em abstract coset}) is a pair $(H,[\;])$ consisting of a non-empty set $H$ and a ternary operation $[\;]\colon H^3\to H$ satisfying, for all $a,b,c,d,e\in H$,
	$$
	[[a,b,c],d,e]=[a,b,[c,d,e]] \quad\text{(para-associativity)}
	$$
	and
	$$
	b=[a,a,b]=[b,a,a]\quad\text{(identity law)}.
	$$
\end{defn}

Heaps are best understood as the remaining structure of a group when we have suppressed the knowledge of the identity. Concretely, there is a correspondence between heaps with a distinguished element and groups given as follows.

If $(H,\cdot)$ is a group, then $[x,y,z]=xy\inv z$ produces a heap operation on $H$.
Conversely, if $(H,[\;])$ is a heap and $e\in H$ is fixed, then $x\cdot y=[x,e,y]$ defines a group operation on $H$, with respect to which $e$ is the identity and the inverse is given by $x\inv=[e,x,e]$. Moreover, these constructions are inverses of each other.

As is easy to see, the subheaps of $(H,[\;])$, that is  non-empty subsets closed under the heap operation, are simply left or, equivalently, right-cosets of subgroups of the group $(H, \cdot)$.

Note that, if $H$ is a heap and $\cdot$ and $*$ are the two group operations $x\cdot y=[x,e,y]$ and $x*y=[x,a,y]$ corresponding to the choices $e$ and $a$ of identity element and $a\inv=[e,a,e]$ denotes the inverse of $a$ in $(H,\cdot)$, then the para-associativity and identity laws give us
\[\begin{split}
x*y
&=[x,a,y]=[x,a,[e,e,y]]=[[x,a,e],e,y]= [[[x,e,e],a,e],e,y]\\
&= [[x,e,[e,a,e]],e,y]=x\cdot a\inv \cdot y.
\end{split}\]
It follows from this that the mappings
$$
\lambda_{a\cdot}\colon x\mapsto a\cdot x\quad\text{ and }\quad \rho_{\cdot a}\colon x\mapsto x\cdot a
$$
are both isomorphisms between the groups $(H, \cdot)$ and $(H, *)$.

With this, we can also describe the isomorphisms between two heaps $(H,[\;])$ and $(G,[\;])$. Indeed, letting $(H,\cdot)$ and $(G,\star)$ be the groups corresponding to some choices of identity, every isomorphism $\alpha\colon (H,[\;])\to (G,[\;])$ can be written as
$$
\alpha=\lambda_{a\star}\circ \beta=\rho_{\star b}\circ \delta= \gamma\circ \lambda_{c\cdot}= \sigma\circ \rho_{\cdot d}
$$
for some $a,b\in G$, $c,d\in H$ and group isomorphisms $\beta,\gamma, \delta, \sigma\colon (H,\cdot)\to (G,\star)$.

A {\em topological heap} is simply a heap $(H,[\;])$ in which $H$ is a topological space and the heap operation is continuous. Observe that, in this case, for any choice of $e\in H$, the group operation and the inverse operation defined above are both continuous and so $(H,\cdot)$ is a topological group. Conversely, a topological group gives rise to a topological heap.

Another way of looking at heaps is by considering the graph $\mathcal G_{[\;]}$ of the heap operation. That is, if $(H,[\;])$ is a heap and we consider the group operations corresponding to some choice of identity element $e\in H$, then
\[\begin{split}
\mathcal G_{[\;]}
&=\{(x,y,z,u)\mid [x,y,z]=u\}\\
&=\{(g,ga,gb,ga^{-1}b)\mid g,a,b\in H\}\\
&=\{(g,ag,bg,a^{-1}bg)\mid g,a,b\in H\}.
\end{split}\]
For example, to see the equality of the last two sets, note that, for $g,a,b\in H$,
$$
(g,ga,gb,ga^{-1}b)=(g, (gag\inv)g, (gbg\inv)g, (gag\inv)^{-1}(gbg\inv)g).
$$

\begin{lem}
	Suppose  $G$ and $H$ are Polish heaps homeomorphically and densely embedded in Polish spaces $X$ and $Y$ respectively and that $\ov {\mathcal G_{[\;]_G}}$ and $\ov{\mathcal G_{[\;]_H}}$ are the closures of the graphs of the heap operations inside $X^4$ and $Y^4$ respectively. Assume that 
	$$
	\phi\colon (X, \ov {\mathcal G_{[\;]_G}})\to (Y, \ov {\mathcal G_{[\;]_H}})
	$$ 
	is a homeomorphic isomorphism. Then $\phi$ maps $G$ onto $H$ and thus restricts to an isomorphism of topological heaps.
\end{lem}

\begin{proof}
Observe first that, being Polish, $G$ and $H$ are dense $G_\delta$ and thus comeagre subsets of $X$ and $Y$ respectively. It follows that also $\phi\inv(H)$ is comeagre in $X$ and thus that $\phi[G]\cap H\neq\emptyset$. Note also that, as the heap operations are continuous, their graphs are closed subsets of $G^4$ and $H^4$ respectively, whence $\ov{\mathcal G_{[\;]_G}}\cap G^4=\mathcal G_{[\;]_G}$ and $\ov{\mathcal G_{[\;]_H}}\cap H^4=\mathcal G_{[\;]_H}$.

So fix $e\in G\cap \phi\inv (H)$. We will view $G$ and $H$ as Polish topological groups with $e$ and $\phi(e)$ as identities, that is, so that the group operations are given by $a\cdot_Gb=[a,e,b]_G$ and $x\cdot_H y=[x,\phi(e),y]_H$ respectively. 
	
Observe first that, for all $x,y,z\in G$ with $x,y,z, [x,y,z]\in \phi\inv (H)$, we have $(x,y,z,[x,y,z])\in {\mathcal G_{[\;]_G}}$ and thus $(\phi(x),\phi(y),\phi(z),\phi([x,y,z]))\in \ov{\mathcal G_{[\;]_H}}\cap H^4=\mathcal G_{[\;]_H}$, whence 
$$
\phi([x,y,z])=[\phi(x),\phi(y),\phi(z)].
$$

Observe now that, for $a,b\in G$, the set of $g\in G$ so that $g,ga, gb, [g,a,b]=ga\inv b\in \phi\inv (H)$ is comeagre, as $G\cap \phi\inv (H)$ is comeagre in $G$ and right translation by $a$, $b$ and $a\inv b$ are homeomorphisms of $G$. Therefore, given $a,b\in G\cap \phi\inv(H)$, we can find a sequence $g_n\in G$ converging to $e$ so that $g_n,g_na, g_nb, [g_n,a,b]\in \phi\inv (H)$ for all $n$. It thus follows by continuity of the heap operations in $G$ and $H$ and by continuity of $\phi$ that
\[\begin{split}
\phi(a\inv b)
&=\phi([e,a,b])\\
&=\lim \phi([g_n,a, b])\\
&=\lim [\phi(g_n), \phi(a),\phi(b)]\\
&=[\phi(e), \phi(a), \phi(b)]\\
&=\phi(a)\inv \phi(b).
\end{split}\]
	In other words, if $a,b\in G\cap \phi\inv(H)$, then $\phi(a\inv b)=\phi(a)\inv \phi(b)\in H$ and so $a\inv b\in G\cap \phi\inv (H)$. In particular, as $G\cap \phi\inv(H)$ is comeagre in $G$, 
	$$
	G= \big(G\cap \phi\inv(H)\big)\inv \cdot \big(G\cap \phi\inv(H)\big)\subseteq G\cap \phi\inv (H),
	$$ 
	so $\phi(G)\subseteq H$ and $\phi\colon G\til H$ is a homomorphism. It follows that $\phi[G]$ is a comeagre subgroup of $H$, which thus can have only one coset in $H$, i.e., $\phi[G]=H$. 
	
	Again, as $\phi$ maps $\ov{\mathcal G_{[\;]_G}}\cap G^4={\mathcal G_{[\;]_G}}$ to $ \ov{\mathcal G_{[\;]_H}}\cap H^4={\mathcal G_{[\;]_H}}$, we see that it is an isomorphism of the topological heaps $(G,[\;]_G)$ and $(H,[\;]_H)$.
\end{proof}

Now, if $G$ is a Polish heap and $\cdot$ and $*$ are the group operations corresponding to different choices $e$ and $a$ of identity, then $(G,\cdot)$ and $(G,*)$ are isomorphic as topological groups.  Moreover, an isomorphism is given by the map $\lambda_{a\cdot}\colon x\mapsto a\cdot x$.

Now, if we equip each of $(G,\cdot)$ and $(G,*)$ with their Roelcke uniformities, then, since left-multiplication  by the $\cdot$ inverse $a\inv$ of $a$ is a uniform homeomorphism of $(G,\cdot)$, we see that
$$
(G,\cdot)\overset{\lambda_{a\inv \cdot}}\longrightarrow (G,\cdot )\overset{\lambda_{a\cdot}}\longrightarrow (G,*)
$$
is a series of uniform homeomorphisms composing to the identity ${\rm id}_G$ on $G$. In other words, the Roelcke uniformity on $G$ (and not just its uniform homeomorphism type) is independent of the choice of group identity. We call this the Roelcke uniformity of the heap $G$. In particular, $G$ is said to be {\em Roelcke precompact} if some or equivalently all induced groups are Roelcke precompact. 

Now, if $G$ is a Roelcke precompact Polish heap, we let $\ov G$ denote its Roelcke completion. Also, let $\ov{\mathcal G_{[\;]_G}}$ denote the closure of the graph of the heap operation inside $\ov G^4$.

\begin{prop}
Let $G$ and $H$ be Roelcke precompact Polish heaps. Then $G$ and $H$ are isomorphic if and only if the compact metrizable structures $(\ov G,\ov{\mathcal G_{[\;]_G}})$ and $(\ov H, \ov{\mathcal G_{[\;]_H}})$ are homeomorphically isomorphic.
	
Moreover, every homeomorphic isomorphism between $(\ov G,\ov{\mathcal G_{[\;]_G}})$ and $(\ov H, \ov{\mathcal G_{[\;]_H}})$  extends an isomorphism of the heaps $G$ and $H$.
\end{prop}

\begin{proof}
By the preceding lemma, it suffices to note that every topological isomorphism between $G$ and $H$ extends to a homeomorphic isomorphism of  $(\ov G,\ov{\mathcal G_{[\;]_G}})$ and $(\ov H, \ov{\mathcal G_{[\;]_H}})$. But this is trivial since an isomorphism of $G$ and $H$ will be a uniform homeomorphism with respect to the Roelcke uniformities.
\end{proof}




\section{Non-Archimedean groups}
Our next step is to consider  the non-Archimedean versions of the above classes of groups, i.e., those possessing a neighborhood basis at the identity consisting of open subgroups. As it turns out, when restricting to uncountable groups, the Alexandrov, respectively, the Roelcke compactifications will be homeomorphic to Cantor space $2^\mathbb N$, whence we obtain a better upper bound for the complexity of isomorphism.

First, it is trivial to see that the Alexandrov compactification of an uncountable non-Archimedean locally compact Polish group is perfect and zero-dimensional, whence, by a theorem of L. E. J. Brouwer,  homeomorphic to Cantor space. Similarly, if $G$ is a non-Archimedean Polish group, then $G$ has a compatible left-invariant metric $d$ taking values only in $\{0\}\cup \{\frac 1n\}_{n\in\mathbb N}$, whereby the Roelcke metric $d_\wedge$ also takes values in the same set. As $\{\frac 1n\}_{n\in \mathbb N}$ is discrete,  it follows that the Roelcke completion of $G$ has a compatible metric whose positive values lie in the same set, whereby the completion must be zero-dimensional. Thus, by the theorem of Brouwer, if $G$ is Roelcke precompact and non-Archimedean, the Roelcke completion is homeomorphic to Cantor space.

Second, if $\mathcal M=(M, (R^\mathcal M)_{R\in \mathcal L})$ is a compact metrizable $\mathcal L$-structure so that $M$ is homeomorphic to $2^\mathbb N$, let 
$$
\mathfrak A=({\rm clopen}(M), {\bf 0}^\mathfrak A, {\bf 1}^\mathfrak A,  \neg^\mathfrak A, \wedge^\mathfrak A, \vee^\mathfrak A, (R^\mathfrak A)_{R\in \mathfrak L})
$$ 
be the countable first-order structure, whose universe is the algebra, ${\rm clopen}(M)$, of clopen subsets of $M$ and where ${\bf 0}^\mathfrak A, {\bf 1}^\mathfrak A,  \neg^\mathfrak A, \wedge^\mathfrak A, \vee^\mathfrak A$ are respectively $\{\emptyset\}$, $\{M\}$ and the graphs of the  functions: complementation, intersection and union. Moreover, for each $R\in \mathcal L$, $$
R^\mathfrak A=\{(a_1,a_2,\ldots, a_n)\in {\rm clopen}(M)^n\mid R^\mathcal M\cap (a_1\times \ldots \times a_n)=\emptyset\},
$$
which simply codes the complement of the closed set $R^\mathcal M$ in $M$. Thus, as $M$ is homeomorphic to Cantor space, the algebra of clopen set is a countable atomless boolean algebra and so, by Stone duality, we see that $\mathfrak A$ is a complete invariant for $\mathcal M$, that is, that two compact metrizable structures $\mathcal M$ and $\mathcal M'$ with universes homeomorphic to Cantor space are homeomorphically isomorphic if and only if the associated countable structures $\mathfrak A$ and $\mathfrak A'$ are isomorphic.

\begin{lem}
Assume $\mathcal L$ is a countable relational language and define the expanded language $\mathcal L'=\mathcal L\cup \{{\bf 0}, {\bf 1}, \neg, \wedge, \vee\}$. Let also $\mathfrak C_\mathcal L=\{(M, (R^\mathcal M)_{R\in \mathcal L})\in \mathfrak K_\mathcal L\mid M\cong 2^\mathbb N\}$ be the Borel set of compact metrizable $\mathcal L$-structures with universe homeomorphic to $2^\mathbb N$. Then there is a Borel map 
$$
\mathcal M\in \mathfrak C_\mathcal L\;\;\mapsto\;\; \mathfrak A_\mathcal M\in {\rm Mod}(\mathcal L')
$$ 
so that $\mathcal M\cong \mathcal M'$ if and only if $ \mathfrak A_\mathcal M\cong  \mathfrak A_{\mathcal M'}$.
\end{lem}

\begin{proof}
Fix a countable basis $\{U_n\}_{n\in \mathbb N}$ for the topology on $\mathcal Q$ closed under finite intersections and unions. Observe first that the collection $\mathcal  C\subseteq K(\mathcal Q)$ consisting of compact subsets $M$ homeomorphic to Cantor space is $G_\delta$. Also, if $M\in \mathcal C$ and $C\subseteq M$ is a clopen subset, then, since the basis is closed under finite unions,  there are $p,q$ so that $C=M\cap U_p$ and $X\setminus C=M\cap U_q$. Again, for fixed $p,q$, the conditions $M\cap U_p\cap U_q=\emptyset$ and $M\subseteq U_p\cup U_q$ are Borel in the variable $M\in \mathcal C$.  It follows that the condition
$$
M\cap U_p \text{ is clopen in } M
$$
is also Borel in $M\in \mathcal C$, whereby the same holds for the condition
$$
M\cap U_p \text{ is clopen in } M \;\;\&\;\; M\cap U_q \text{ is clopen in } M\;\;\&\;\; M\cap U_p=M\cap U_q.
$$
It follows that the set of $p$ so that  $M\cap U_p$ is clopen and $M\cap U_p\neq M\cap U_q$ for all $q<p$ can be computed in a Borel manner from $M$. Identifying  ${\rm clopen}(M)$ with the set these $p$, we obtain a representation of the clopen algebra in which the algebra operations are similarly Borel in $M$.

In other words, the complete invariant 
$$
\mathfrak A=({\rm clopen}(M), {\bf 0}^\mathfrak A, {\bf 1}^\mathfrak A,  \neg^\mathfrak A, \wedge^\mathfrak A, \vee^\mathfrak A, (R^\mathfrak A)_{R\in \mathfrak L})
$$ 
has an isomorphic realisation with universe included in $\mathbb N$ that can be computed in a Borel manner from $\mathcal M\in \mathfrak C_\mathcal L$. Reenumerating the universe as $\mathbb N$, we obtain a complete invariant $ \mathfrak A_\mathcal M\in {\rm Mod}(\mathcal L')$.
\end{proof}

Separating into the discrete and uncountable case, we arrive at the following conclusion.
\begin{prop}\label{non-Archimedean}
The relation of isomorphism between locally compact or Roelcke precompact non-Archimedean Polish groups is Borel reducible to isomorphism between countable structures, i.e., to the complete orbit equivalence relation $\mathsf E_{S_\infty}$ induced by an action of $S_\infty$.
\end{prop}


\section{An application to model theory}
The Borel reducibility theory for equivalence relations has played an important role in understanding the isomorphism relation between countable structures in a countable language. We note here, however, a consequence for the \emph{bi-interpretability} relation for $ \omega $-categorical structures. We recall that the automorphism group of any $ \omega $-categorical structure is Roelcke-precompact, and by \cite{ahlbrandt1986quasi}, two such structures are bi-interpretable if and only if their automorphism groups are isomorphic as topological groups. Combining this with the results of the last section gives,

\begin{prop} \label{bi-interpretability}
	The bi-interpretability relation for $ \omega $-categorical structures is classifiable by countable structures.
\end{prop}

To make sense of this proposition, we first let $ \widehat{\mathcal{L}} $ be a language with infinitely many predicate symbols of each arity. The space $ \operatorname{Mod}(\widehat{\mathcal{L}}) $ of countable structures in the language $ \widehat{\mathcal{L}} $ we view as parameterizing all countable structures up to bi-interpretability, by associating a structure $ \mathcal{M} =  \{ a_{1}, a_{2}, \dots \} $ in a language $ \mathcal{L} $ with the structure in $ \operatorname{Mod}(\widehat{\mathcal{L}}) $ obtained by taking any arity-preserving injection $ j\colon \mathcal{L} \to \widehat{\mathcal{L}} $ and letting $ R(k_{1},\dots,k_{n}) $ hold for $ R \in \widehat{\mathcal{L}} $ if and only if $ R \in j[\mathcal{L}] $ and $ \mathcal{M} \models j^{-1}(R)(a_{k_{1}},\dots,a_{k_{n}}) $.

Second, we let $ \mathcal{C} \subseteq \operatorname{Mod}(\mathcal{\widehat{L}}) $ consist of the $ \omega $-categorical structures, and note that this set is Borel. For by the theorem of Engeler, Ryll-Nardzewski, and Svenonius, $ M \in \mathcal{C} $ if and only if $ M $ realizes only finitely many $ n $-types for each $ n $. Let 
$ F_{n} = \{\varphi \mid \varphi \text{ is an } \widehat{\mathcal{L}} \text{-formula in } n \text{ free variables} \} $, and let $ S_{n} = 2^{F_{n}} $, with compatible metric $ d $. For $ \bar{a} \in \omega^{n} $, let 
$ f_{\bar{a}}\colon \operatorname{Mod}(\widehat{\mathcal{L}}) \to S_{n} $ be
\[ f_{\bar{a}}(M)(\varphi) =
\begin{cases}
1 & \mid M \models \varphi(\bar{a}) \\
0 & \mid M \not\models \varphi(\bar{a}) 
\end{cases}
\]
and as $\operatorname{Mod}(\varphi,\bar{a}) = \{M \in \operatorname{Mod}(\widehat{\mathcal{L}}) \mid M \models \varphi(\bar{a})  \} $ is Borel for any tuple $ \bar{a} $ and $ \widehat{\mathcal{L}} $-formula (in fact, $ \mathcal{L}_{\omega_{1}\omega} $-formula), $ \varphi $, the functions $ f_{\bar{a}} $ are Borel-measurable.

Then \[ M \in \mathcal{C} \text{ iff }  \forall n, \exists m, \forall \bar{a}, \bar{b} \in \omega^{n}, \left[f_{\bar{a}}(M) = f_{\bar{b}}(M) \vee d(f_{\bar{a}}(M),f_{\bar{b}}(M)) \geqslant \frac{1}{m} \right] \]
so $ \mathcal{C} $ is Borel.

Let $ \cong $, $ \equiv $, and $ \approx $ denote respectively the relations of isomorphism, elementary equivalence, and bi-interpretability on $ \operatorname{Mod}(\widehat{\mathcal{L}}) $. Observe that $ \mathcal{C} $ is an invariant set for each relation. 

\begin{proof}[Proof of Proposition \ref{bi-interpretability}]
	First, we note that $ \equiv $ is a smooth equivalence relation on $ \operatorname{Mod}(\widehat{\mathcal{L}}) $. (The collection $ \{\operatorname{Mod}(\sigma) \mid \sigma \text{ is a sentence} \} $ is a countable Borel separating family.) In particular, $ \equiv $, and so $ \equiv \upharpoonright \mathcal{C} $, is a Borel equivalence relation. But $ \equiv \upharpoonright \mathcal{C} $ is $\cong \upharpoonright \mathcal{C} $, or in other words, the action of $ S_{\infty} $ which induces the isomorphism relation has a Borel orbit equivalence relation. Therefore, by \cite{becker1996descriptive} Theorem 7.1.2, the function $ \mathcal{C} \to F(S_{\infty}) $ taking $ M \mapsto \operatorname{Stab}(M) $ is Borel-measurable. But $ \operatorname{Stab}(M) = \operatorname{Aut}(M) $, so this function is a reduction to isomorphism of non-Archimedean Roelcke-precompact Polish groups, and hence classifiable by countable structures.
\end{proof}


\section{Comparing classes of structures}
As it turns out, the viewpoint of abstract topological or metric structures is useful when comparing complexities of naturally occurring classification problems. So let us consider some of the various types of structures that appear in the literature.  For this, fix a countable relational language $\mathcal L$.

First, a {\em Polish metric $\mathcal L$-structure} is a tuple $\mathcal M=(M,d,(R^\mathcal M)_{R\in \mathcal L})$, where $(M,d)$ is a separable, complete metric space and each $R^\mathcal M$ is a closed subset of $M^{\alpha(R)}$. 
In case that $(M,d)$ is not only separable and complete, but actually compact, we say that $\mathcal M$ is a {\em compact metric $\mathcal L$-structure}. Similarly, if $d$ is a proper metric or if $M$ is locally compact, $\mathcal M$ is called a {\em proper metric $\mathcal L$-structure}, respectively a {\em locally compact metric $\mathcal L$-structure}.

We also define the relations of {\em isometric isomorphism}, {\em bi-Lipschitz isomorphism}, {\em uniformly homeomorphic isomorphism} and {\em homeomorphic isomorphism} between Polish metric structures. Namely, an isomorphism $\phi\colon M\til N$ between the underlying discrete structures $(M,(R^\mathcal M)_{R\in \mathcal L})$  and $(N,(R^\mathcal N)_{R\in \mathcal L})$ will be an isometric, bi-Lipschitz, uniformly homeomorphic or homeomorphic isomorphism if, respectively, $\phi$ and $\phi^{-1}$ are isometries, Lipschitz, uniformly continuous or continuous. 
Evidently, any  homeomorphic isomorphism between compact metric structures is automatically a uniformly homeomorphic isomorphism. In general, however, all these notions differ.

We may parametrize Polish metric structures by identifying them with subsets of the Urysohn metric space equipped with a set of closed relations. As all the classes of spaces are Borel, this provides a Borel parametrisation of each of these classes of structures. Moreover, except possibly for the case of homeomorphic isomorphism of Polish metric structures, this gives rise to an analytic equivalence relation. For example, to see that the relation of uniformly homeomorphic isomorphism of Polish metric structures is analytic, observe that two Polish metric $\mathcal L$-structures $\mathcal M$ and $\mathcal N$ are uniformly homeomorphically isomorphic if and only if
\[\begin{split}
\exists & (x_n)\in M\; \exists (y_n)\in N\; \Big( \overline{\{x_n\}}=M\;\&\; \overline{\{y_n\}}=N\;\&\\
& \forall R\in \mathcal L\; R^{\mathcal M}=\overline{R^{\mathcal M}\cap \{x_n\}^{\alpha(R)}}\;\&\; R^{\mathcal N}=\overline{R^{\mathcal N}\cap \{y_n\}^{\alpha(R)}}\\
& \&\; \forall \epsilon>0\;\exists \delta>0\;\big( d(x_i,x_j)<\delta \to d(y_i,y_j)<\epsilon\big)\\
& \& \; \forall \epsilon>0\;\exists \delta>0\; \big(d(y_i,y_j)<\delta \to d(x_i,x_j)<\epsilon\big)\\
& \& \;\forall R\in \mathcal L\; \big((x_{i_1}, \ldots, x_{i_{\alpha(R)}})\in R^{\mathcal M}\leftrightarrow (y_{i_1}, \ldots, y_{i_{\alpha(R)}})\in R^{\mathcal M}\big)\Big).
\end{split}\]

The only relation standing out is homeomorphic isomorphism between Polish metric structures of which the exact complexity is still unknown.

Let us also mention that the choice of considering only relational languages is no real restriction. Indeed, if $F$ is an $n$-ary function symbol in $\mathcal L$ and $F^\mathcal M$ is an interpretation of $F$ as a continuous function on $\mathcal M$, then we could simply introduce a symbol for the graph of $F^\mathcal M$, which is a closed subset of $M^{n+1}$. Moreover, isomorphisms preserving the graphs will also necessarily preserve the functions and vice versa.

One way of phrasing the statement of Proposition \ref{bireducible to Egrp} is that, for any countable $ \mathcal{L} $, homeomorphic isomorphism between compact $ \mathcal{L} $-structures is Borel bi-reducible with homeomorphism between compact metric spaces (i.e., homeomorphic isomorphism  of compact structures in the empty language). That is, the imposition of the additional structure from $\mathcal L$ does not actually make the corresponding relation more complicated from the perspective of Borel reducibility. As it turns out, this same can be said for all of the relations
\begin{itemize}
\item homeomorphic isomorphism,
\item uniformly homeomorphic isomorphism,
\item bi-Lipschitz isomorphism,
\item isometric isomorphism,
\end{itemize}
when considered between compact metric structures. Moreover, with the possible exception of homeomorphic isomorphism, this remains true for the class of Polish metric structures.

Indeed, for isometric isomorphism in class of Polish metric structures, this was established in \cite{elliott2013isomorphism}. And as bi-Lipschitz isomorphism and uniform homeomorphism of Polish metric spaces are both complete analytic equivalence relations \cite{ferenczi2009complexity}, the same is true a posteriori for the corresponding relations for Polish metric $ \mathcal{L} $-structures. Moreover, as any homeomorphism between compact metric spaces is a uniform homeomorphism, the claim for uniformly homeomorphic isomorphism of compact metric structures holds by Proposition \ref{bireducible to Egrp}.

So it remains to see that this is true for isometry and bi-Lipschitz isomorphism of compact metric spaces, a smooth and a complete $ K_{\sigma} $ equivalence relation, respectively.

In the following, we consider a fixed language $ \mathcal{L} = \{R_{1},R_{2},\dots \} $, where $R_k$ is a relation symbol of arity $k$. Since $\mathcal L$ has symbols of unbounded arity, isometric isomorphism of arbitrary compact metric structures is easily reducible to isometric isomorphism of compact metric $\mathcal L$-structures.  Also, if $(X,d)$ is a compact metric space, we equip each power $X^n$ with the compatible metric $d_\infty$ defined by
$$
d_\infty\big((x_1,\ldots, x_n),(y_1,\ldots, y_n)\big).
$$

\begin{thm} \label{isometry smooth}
Isometric isomorphism of compact metric $ \mathcal{L} $-structures is smooth.
\end{thm}

\begin{proof}
For each tuple $\zeta = (S_{1},S_{2},\dots,S_{n}) $ with $S_{k} \subseteq  \{1,\dots,n\}^k$ and every compact metric structure $\mathcal X=(X,(R_{k}^{\mathcal X})) $, let 
$$ 
C_{\zeta}(\mathcal X)=\big\{ [d(x_i,x_j)]_{i,j\leqslant n}    \;\big|\;x_i\in X 
\;\; \&\;\; 
 (x_{s_1}, \ldots, x_{s_k})\in R_k^\mathcal X, \forall  k\leqslant n,  s\in S_k \big \}.
$$ 
Clearly, each $ C_{\zeta}(\mathcal X) $ is a compact subset of $\mathbb R^{n^2}$,  invariant for the isometric isomorphism type of $\mathcal X$. 

Conversely, suppose that $\mathcal X=(X,(R_{k}^{\mathcal X})) $ and $\mathcal Y=(Y,(R_{k}^{\mathcal Y}))$ are two compact metric $\mathcal L$-structures so that $C_{\zeta}(\mathcal X) = C_{\zeta}(\mathcal Y)$ for all $\zeta$. Pick a sequence $ (x_{m})_{m \in \mathbb{N}}$ enumerating a countable dense subset of $X$ so that also the set of tuples $(x_{s_1}, \ldots, x_{s_k})\in R^\mathcal X_k$ is dense in $R^\mathcal X_k$ for each $k$. 

As $C_{\zeta}(\mathcal X) = C_{\zeta}(\mathcal Y)$ for all $\zeta$, for each $ n $, we may find an $n$-tuple, $(y_{1}^{n}, \dots, y_{n}^{n}) \in Y^{n} $ with the same distance matrix $[d(y^n_i,y^n_j)]_{i,j}$ as $(x_{1}, \dots, x_{n}) $ so that
$$
(x_{s_1}, \ldots, x_{s_k})\in R_k^\mathcal X\Rightarrow  (y^n_{s_1}, \ldots, y^n_{s_k})\in R_k^\mathcal Y 
$$
for all $k\leqslant n$ and $s\in \{1,\dots,n\}^k$.

Then, by a diagonalization and passing to limits, we may find an infinite sequence $(y_1, y_2, \ldots)$ in $Y$ with the same distance matrix as $(x_1,x_2,\ldots)$ and still satisfying 
$$
(x_{s_1}, \ldots, x_{s_k})\in R_k^\mathcal X\Rightarrow  (y_{s_1}, \ldots, y_{s_k})\in R_k^\mathcal Y 
$$
for all $k$ and $s\in \mathbb N^k$. Extending the map $x_i\mapsto y_i$  to the completion determines an isometric embedding $ f\colon X \to Y $. Moreover, for every $k$, $f$ maps $R^\mathcal X_k$ into $R^\mathcal Y_k$.

A symmetric argument produces an isometric embedding $ g\colon Y \to X $ that maps each 
$R^\mathcal Y_k$ into $R^\mathcal X_k$. It thus follows that $g\circ f$ is an isometric embedding of the compact metric space $X$ into itself and hence must be surjective. Similarly, for each $k$, $g\circ f\colon R^\mathcal X_k\to R^\mathcal X_k$ is an isometric embedding with respect to the product metric $d_\infty$ on $X^n$ and thus again surjective.    

It therefore follows that $f$ is an isometric isomorphism of $\mathcal X$ with $\mathcal Y$ and so the sequence $\big(C_\zeta(\mathcal X)\big)_\zeta$ is a complete invariant for isometric isomorphism.
\end{proof}

Just as the family of distance matrices corresponding to finite subsets of a compact metric space is a complete invariant for the isometry type and hence induces a smooth equivalence relation  as shown by M. Gromov (Propositions 3.2 and 3.6 \cite{gromov}), Theorem \ref{isometry smooth} demonstrates that the \emph{isometric isomorphism} type of a \emph{compact metric structure} is captured by a slightly modified invariant that accounts for the distance matrices for finite sets satisfying constraints given by the relational structure.

The relation of bi-Lipschitz isomorphism between compact metric spaces was shown to be bi-reducible with a complete $ K_{\sigma} $ equivalence relation in \cite{rosendal2005cofinal}. In particular, it is reducible to a $ K_{\sigma} $ relation; one sees this by associating, to a given metric space, a more complicated invariant derived from distance matrices of finite sets in $ X $. Next we augment this construction in a manner similar to the proof of Theorem \ref{isometry smooth}. Consequently, for any countable $ \mathcal{L} $, the relation of \emph{bi-Lipschitz} isomorphism between compact metric $ \mathcal{L} $-structures is reducible to a $ K_{\sigma} $, and is therefore bi-reducible with the complete $ K_{\sigma} $ equivalence relation.

\begin{thm}\label{bi-Lipschitz}
The relation of bi-Lipschitz isomorphism between compact metric $ \mathcal{L} $-structures is bi-reducible with the complete $K_{\sigma}$ equivalence relation.
\end{thm}

\begin{proof}
Since in \cite{rosendal2005cofinal} this was already proved for the empty language $\mathcal L=\emptyset$, it suffices now to consider the other extreme, namely, when $\mathcal L=\{R_1, R_2, \ldots\}$ where each $R_k$ has arity $k$. We must show that bi-Lipschitz isomorphism between compact metric $ \mathcal{L} $-structures is Borel reducible to a $K_\sigma$ equivalence relation.
For $\alpha>0$, an $\alpha$-perturbation of a matrix $[a_{ij}]$ is  defined to be a matrix $[b_{ij}]$ so that $\frac 1\alpha a_{ij}\leqslant b_{ij}\leqslant \alpha a_{ij}$ for all $i,j$.

For every tuple $ \zeta=(S_k, r_k, t_{p,k})_{p\leqslant k\leqslant n}$ with $S_k\subseteq \{1, \dots, n\}^{k} $ and $r_k,t_{p,k}\in \mathbb Q_+$, and  every compact metric $\mathcal L$-structure $\mathcal X= (X,d,(R_{k}^{\mathcal X})) $, let $ D_{\zeta} (\mathcal X)$ denote the compact set  of all $ n $-tuples $ (x_{1}, \dots, x_{n}) \in X^{n} $ satisfying
\begin{enumerate}
\item $ \{x_{1}, \dots, x_{k}\} $ is  $ r_k$-dense in $ X $ for all $k\leqslant n$,
\item $(x_{s_1}, \ldots, x_{s_k})\in R_k^\mathcal X$ for all $s=(s_1, \ldots, s_k)\in S_k$, and
\item the set  $\big\{ (x_{s_1}, \ldots, x_{s_p})\;\big|\; s\in S_p\cap \{1,\ldots, k\}^p \big\}$ is  $ t_{p,k}$-dense in $R_p^\mathcal X$ for  all $p\leqslant k\leqslant n$ (with respect to the $ d_{\infty} $ metric).	
\end{enumerate}
Let also $ E_{\zeta,\alpha}(\mathcal X) \subseteq \mathbb{R}^{n \times n}  $ denote the compact set of all $\alpha$-perturbations of distance matrices $[d(x_i, x_j)]$ of tuples $(x_1,\ldots, x_n)\in D_{\zeta}(\mathcal X)$.

We claim that two compact metric $\mathcal L$-structures $\mathcal X$ and $\mathcal Y$ are  bi-Lipschitz isomorphic if and only if
\[
 \exists c \in \mathbb{Q}_{+} \; \forall \zeta,\alpha \quad E_{\zeta,\alpha}(\mathcal X) \subseteq E_{c\zeta, c\alpha}(\mathcal Y), 
\]
where $ c \zeta = (S_k, cr_k, ct_{p,k})_{p\leqslant k\leqslant n}$. As this relation is essentially $ K_{\sigma} $ (see Example (ii) below Definition 18 \cite{rosendal2005cofinal}),  this claim establishes the theorem.
	
So suppose $ f\colon  X \to Y $ is a bi-Lipschitz isomorphism with bi-Lipschitz constant $c\in \mathbb Q_+$. Then it is straightforward to verify that $E_{\zeta,\alpha}(\mathcal X) \subseteq E_{c\zeta, c\alpha}(\mathcal Y)$ for all tuples $\zeta$ and $\alpha\in \mathbb Q_+$.

Conversely, suppose that $c>0$ is so that
$$
E_{\zeta,\alpha}(\mathcal X) \subseteq E_{c\zeta, c\alpha}(\mathcal Y)
$$
for all $\zeta$,  $\alpha$. Pick a sequence $ (x_{m}) $ enumerating a dense subset of $ X $ so that, moreover, the set of tuples $(x_{s_1}, \ldots, x_{s_k})\in R_k^\mathcal X$ is dense in $R_k^\mathcal X$ for all $k$.
For all $k\leqslant n$, let 
$$ 
S_k^n = \big\{s\in \{1, \dots, n \}^{k} \;\big|\; (x_{s_1}, \ldots, x_{s_k})\in R_k^\mathcal X\big\}.
$$  
Fix  also $r_n\in \mathbb{Q}_+ $ so that $\{x_{1}, \dots, x_{n}\} $ is  $r_n$-dense in $ X $ and pick $t_{k,n}\in \mathbb Q_+$ so that the set  
$$
\big\{ (x_{s_1}, \ldots, x_{s_k})\;\big|\; s\in \{1, \dots, n \}^{k}\;\&\; (x_{s_1}, \ldots, x_{s_k})\in R_k^\mathcal X\big\}
$$ 
is  $ t_{k,n}$-dense in $R_k^\mathcal X$, but so that they are not $\frac {r_n}2$-dense, respectively $\frac{t_{k,n}}2$-dense (or, in case these sets exhaust $X$ or $R_k^\mathcal X$, we let $r_n=0$, respectively $t_{k,n}=0$). We note that, for $p\leqslant k\leqslant n$, we have $S^k_p=S^n_p\cap \{1,\ldots, k\}^p$, so 
$$
S_p^p\subseteq S_p^{p+1}\subseteq S_p^{p+2}\subseteq\ldots\subseteq \bigcup_nS^n_p=\{s\in \mathbb N^p\mid (x_{s_1},\ldots, x_{s_p})\in R_p^\mathcal X\}.
$$
Observe also that $\lim_{n\to\infty}r_n=0$ and $\lim_{n\to\infty}t_{k, n}=0$ for all $k$.

Set $\zeta_n=(S_k^n,r_k, t_{p,k})_{p\leqslant k\leqslant n}$ and note that $(x_1, \ldots, x_n)\in D_{\zeta_n}(\mathcal X)$ and hence
$$
[d(x_i,x_j)]_{i,j\leqslant n}\in E_{\zeta_n,1}(\mathcal X) \subseteq E_{c\zeta_n, c}(\mathcal Y)
$$
for all $n$.

So let $ (y^{n}_{1}, \dots, y^{n}_{n}) $ be a tuple in $Y$ witnessing that $[d(x_i,x_j)]_{i,j\leqslant n}\in  E_{c\zeta_n, c}(\mathcal Y)$, i.e., 
\begin{enumerate}
\item $\frac 1cd(x_i,x_j)\leqslant d(y^n_i, y^n_j)\leqslant c d(x_i,x_j)$ for all $i,j\leqslant n$,
\item $ \{y^n_{1}, \dots, y^n_{k}\} $ is  $ cr_k$-dense in $Y$ for all $k\leqslant n$,
\item $(y^n_{s_1}, \ldots, y^n_{s_k})\in R_k^\mathcal Y$ for all $s\in S_k^n$, and
\item the set  $\big\{ (y^n_{s_1}, \ldots, y^n_{s_p})\;\big|\; s\in S_p^n\cap\{1,\ldots, k\}^p\big\}$ is  $ct_{p,k}$-dense in $R_p^\mathcal Y$ for all $p\leqslant k\leqslant n$.	
\end{enumerate}

By diagonalization, we may produce an increasing sequence $(n_l)$ so that $y_k=\lim_l y^{n_l}_k$ exists for every $k$. It then easily follows that
\begin{enumerate}
\item $\frac 1cd(x_i,x_j)\leqslant d(y_i, y_j)\leqslant c d(x_i,x_j)$ for all $i,j$,
\item $ \{y_{1}, \dots, y_{k}\} $ is  $ cr_k$-dense in $Y$,
\item $(y_{s_1}, \ldots, y_{s_k})\in R_k^\mathcal Y$ whenever $(x_{s_1}, \ldots, x_{s_k})\in R_k^\mathcal X$, and 
\item the set  $\big\{ (y_{s_1}, \ldots, y_{s_p})\;\big|\; s\in \{1,\ldots, k\}^p\;\&\; (x_{s_1}, \ldots, x_{s_p})\in R_p^\mathcal X\big\}$ is  $ct_{p,k}$-dense in $R_p^\mathcal Y$ for all $p\leqslant k\leqslant n$.	
\end{enumerate}

Since $\lim_{k\to \infty} r_k=\lim_{k\to \infty} t_{p,k}=0$ for all $p$, it follows that $(y_k)$ is a dense sequence in $Y$ and that 
$$
\big\{ (y_{s_1}, \ldots, y_{s_p})\;\big|\; (x_{s_1}, \ldots, x_{s_p})\in R_p^\mathcal X\big\}
$$
is dense in $R^\mathcal Y_p$ for every $p$. The map $x_i\mapsto y_i$ therefore extends to a bi-Lipschitz isomorphism between $X$ and $Y$ with constant $c$, mapping each $R_p^\mathcal X$ surjectively onto $R_p^\mathcal Y$. In other words, $\mathcal X$ and $\mathcal Y$ are bi-Lipschitz isomorphic.
\end{proof}

\bibliographystyle{plain}
\bibliography{cms}

\end{document}